\documentclass[hidelinks]{article}
\usepackage{graphicx} 
\usepackage{hyperref}
\usepackage{amsthm}
\usepackage{thm-restate}
\usepackage{geometry}
\usepackage{amssymb}
\usepackage{amsmath}
\usepackage{cleveref}
\usepackage{comment}
\usepackage{xcolor}
\usepackage[shortlabels]{enumitem}
\usepackage{listings}
\usepackage{blkarray}
\usepackage{lmodern}
\usepackage[english]{babel}
\usepackage{float}
\usepackage{tikz}

\newtheorem{theorem}{Theorem}[section]
\newtheorem{lemma}[theorem]{Lemma}

\theoremstyle{definition}

\newcommand{\cC}{\mathcal{C}}
\newcommand{\cF}{\mathcal{F}}

\newcommand{\cG}{\mathcal{G}}
\newcommand{\cS}{\mathcal{S}}
\newcommand{\cH}{\mathcal{H}}

\newcommand{\eps}{\varepsilon}

\title{Improved bound on the number of cycle sets}

\author{
Rajko Nenadov\thanks{School of Computer Science, University of Auckland, New Zealand. Email: \texttt{rajko.nenadov@auckland.ac.nz}. Research supported by the Marsden Fund of the Royal Society of New Zealand.}
}
\date{}

\begin{document}

\maketitle

\begin{abstract}
The cycle set of a graph $G$ is the set consisting of all sizes of cycles in $G$. Answering a conjecture of Erd\H{o}s and Faudree, Verstra\"{e}te showed that there are at most $2^{n - n^{1/10}}$ different cycle sets of graphs with $n$ vertices. We improve this bound to $2^{n - n^{1/2 - o(1)}}$. Our proof follows the general strategy of Verstra\"{e}te of reducing the problem to counting cycle sets of Hamiltonian graphs with many chords or a large maximum degree. The key new ingredients are near-optimal container lemmata for cycle sets of such graphs.
\end{abstract}

\section{Introduction}

Given a graph $G$ with $n$ vertices, the \emph{cycle set $\cS(G) \subseteq \{3, \ldots, n\}$} is defined as $\ell \in \cS(G)$ if and only if $G$ contains a cycle of size $\ell$. The study of sufficient conditions which guarantee that $\cS(G)$ has certain properties is one of the central topics in graph theory. For example, Dirac's theorem \cite{dirac52} gives a sufficient condition for $n \in \cS(G)$ (that is, $G$ is \emph{Hamiltonian}) in terms of the minimum degree of $G$; Mantel's theorem gives a sufficient condition which ensures $3 \in \cS(G)$ in terms of the number of edges, and so on. Closely related to Dirac's theorem, and many other similar statements about hamiltonicity of graphs, are sufficient conditions which ensure that $G$ is \emph{pancyclic}, that is, $\cS(G) = \{3, \ldots, n\}$ (e.g.~see \cite{bauer90pan,bondy71pan,draganic2024pancyclicity}).

Previous examples focus on either a particular value appearing in $\cS(G)$, or $\cS(G)$ itself being a particular set. Another influential line of research is the study of the size and structure of $\cS(G)$. It is a simple exercise to show that if a graph $G$ has minimum degree $d$, then $|\cS(G)| \ge d - 1$. Sudakov and Verstra\"{e}te \cite{sudakov08cyclelengths} showed that if $G$ is a graph with average degree $d$ and girth at least $2g+1$, then $|\cS(G)| = \Omega(d^g)$. Moreover, they showed that $\cS(G)$ contains a large subset of consecutive values. Milans, Pfender, Rautenbach, Regen, and West \cite{milans12spectra} showed that if $G$ contains a Hamilton cycle and has $n + p$ edges, then $|\cS(G)| = \Omega(\sqrt{p})$. While this result is, in general, optimal, Buci\'c, Gishboliner, and Sudakov \cite{bucic22ham} showed that if $G$ is Hamiltonian with minimum degree at least $3$, then $|\cS(G)| \ge n^{1 - o(1)}$. A celebrated result of Gy\'arf\'as, Koml\'os, and Szemer\'edi \cite{gyarfas84cycles} shows that if $G$ has average degree $d$, then $\sum_{\ell \in \cS(G)} 1 / \ell > \eps \log(d)$. Note that this is really a structural result: if a graph does not have short cycles, it has many long ones. The lower bound was recently strengthened to an almost optimal $(1/2 - o_d(1))\log(d)$ by Liu and Montgomery \cite{hong23cycle}.

The previous list of examples merely touches on the long list of results about cycle sets and is only meant to convey their breadth and depth. For a thorough treatment of the topic, we refer the reader to the survey by Verstra\"{e}te \cite{verstraete16book}. With this in mind, problems of Erd\H{o}s (see \cite[Problem 5]{erdos97problems}; he attributes them to Faudree and himself) of characterising subsets of $\{3, \ldots, n\}$ which are the cycle set of a $n$-vertex graph, or how many different cycle sets of $n$-vertex graphs there are, are  rather natural ones. Erd\H{o}s conjectured that cycle sets form a vanishing fraction of all possible subsets, and this was verified by Verstra\"{e}te \cite{verstraete04cycles}. In particular, he showed that there are at most $2^{n - n^{1/10}}$ subsets of $\{3, \ldots, n\}$ which are the cycle set of a $n$-vertex graph. We improve this bound:

\begin{theorem} \label{thm:main}
    The number of different cycle sets of $n$-vertex graphs is at most $2^{n - \Omega(\sqrt{n} / \log^{3/2}(n))}$.
\end{theorem}

The proof of Theorem \ref{thm:main} follows the overall strategy of Verstra\"{e}te \cite{verstraete04cycles}, which through a series of assumptions reduces the problem to counting cycle sets of Hamiltonian graphs with many chords or a large maximum degree. We improve upon both of these sub-problems using container-type lemmata, presented in Section \ref{sec:containers}. Both of these lemmata crucially rely on Lemma \ref{lemma:compression} which provides a small \emph{fingerprint} for a given set of chords. We postpone the discussion about possible further improvements of Theorem \ref{thm:main} to Section \ref{sec:concluding}.

At present, we do not know how far Theorem \ref{thm:main} is from the truth. The best known lower bound is a construction due to Faudree: Suppose $n$ is even. Given $A \subseteq \{n/2 +1 , \ldots, n\}$, form the graph $G$ by taking a Hamilton path $1, 2, \ldots, n-1, n$ and all the edges $\{1, a\}$ for $a \in A$. Then $\cS(G) \cap \{n/2+1, \ldots, n\} = A$. This shows there are at least $2^{n/2}$ different cycle sets of $n$-vertex graphs. Improving this to $2^{(1 + c)n/2}$, for a constant $c > 0$, would already be interesting.

\section{Definitions}

Throughout the paper we tacitly avoid use of floors and ceilings. All stated inequalities can be made to hold with sufficient margin to accommodate for this. For brevity, we introduce a number of definitions.

\begin{itemize}
\item Given a family of graphs $\cG$, we define $\cF(\cG)$ to be the family of all cycle sets of graphs in $\cG$:
$$
    \cF(\cG) = \{\cS(G) \colon G \in \cG\}.
$$

\item Given a family of cycles $\cC$ in some graph $G$, we define
$$
    \cS(\cC) = \{|C| \colon C \in \cC\}.
$$

\item We refer to a graph $G$ on the vertex set $\{1, \ldots, n\}$ as \emph{labelled}. Given an edge $e \in G$ in a labelled graph, we denote with $(a_e, b_e)$ its endpoints such that $a_e < b_e$. 

\item Suppose $G$ is a labelled graph such that $(1, 2, \ldots, n-1, n, 1)$ forms a Hamilton cycle, which we denote with $H$. We refer to the edges $E(G) \setminus E(H)$ as \emph{chords}.  For two distinct chords $e, e' \in E(G) \setminus E(H)$, there are either one or two cycles in $G$ which use both $e$ and $e'$, with all other edges being in $E(H)$. If there is one such cycle, we denote it with $C(e, e')$. This is the case if either $\{a_e, b_e\} \cap \{a_{e'}, b_{e'}\} \neq \emptyset$, $b_e < a_{e'}$, $b_{e'} < a_e$, or  $a_{e} < a_{e'} \Leftrightarrow b_{e'} < b_e$. Otherwise, we have $a_e < a_{e'} < b_{e} < b_{e'}$ (or the same with $e$ and $e'$ swapped), in which case we define the cycle $C(e, e')$  to be the following:
\begin{itemize}
    \item If $a_e < a_{e'}$, then 
    \begin{equation} \label{eq:cycle_def}
        C(e, e') = a_e \to b_e \searrow a_{e'} \to b_{e'} \nearrow a_{e},
    \end{equation}
    where $b_e \searrow a_{e'}$ means ``go from $b_e$ to $a_{e'}$ along $H$ in the decreasing order of labels``, and $b_{e'} \searrow a_e$ is defined analogously in the increasing order (see Figure \ref{fig:C}).
    
    \item If $a_e > a_{e'}$, set $C(e, e') := C(e', e)$,
\end{itemize}
This definition of $C(e, e')$ will be important at one point in the proof of Lemma \ref{lemma:count_ham}.

\begin{figure}[h!]
\centering
\begin{tikzpicture}[scale=0.8]
    \def\r{3cm}        

    \foreach \i in {1,...,10}{
      \pgfmathparse{360/10*\i}
      \draw (\pgfmathresult : -\r) node[draw=black, circle, fill=black, inner sep=0, minimum size=3pt] (\i){};
    }
    \foreach \i in {1,...,10}{
      \pgfmathparse{360/10*\i}
      \draw (\pgfmathresult : -3.5cm) node {\i};
    }

    \draw[dashed] (0,0) circle (3cm);
    \draw[ultra thick] (2) -- (6);
    \draw[ultra thick] (4) -- (8);
    \draw[ultra thick] (0,0) +(7*360/10:3cm) arc (7*360/10:3*360/10:3cm);
    \draw[ultra thick] (0,0) +(9*360/10:3cm) arc (9*360/10:11*360/10:3cm);
\end{tikzpicture}
\caption{The cycle $C(e,e')$ where $e = (2,6)$ and $e' = (4,8)$.}
\label{fig:C}
\end{figure}

Given a subset $R \subseteq E(G) \setminus E(H)$, we denote with $\mathcal{C}(e, R)$ the set of all cycles coming from interactions of $e$ and chords in $R$:
$$
    \mathcal{C}(e, R) = \{ C(e, e') \colon e' \in R \setminus \{e\} \}.
$$
We denote with $\mathcal{C}(R)$ the set of all pairwise interactions of chords in $R$:
$$
    \mathcal{C}(R) = \{C(e, e') \colon e, e \in R, e \neq e' \}.
$$

\item We frequently use the notion of \emph{short-cutting}. Given an edge $e \in G$ and a cycle $C$ in a labelled $G$, such that vertices $a_e, \ldots, b_e$ appear in that order in $C$, \emph{short-cutting $C$ along $e$} refers to the cycle obtained from $C$ by simply removing vertices labelled $a_{e}+1, \ldots, b_{e}-1$ (this is indeed a cycle as $e$ connects $a_e$ and $b_e$).
\end{itemize}

\section{The Fingerprint lemma}

The following lemma is the key ingredient in the proof of results in Section \ref{sec:containers} which, in turn, are the main new components in the proof of Theorem \ref{thm:main}. We think of the set $F$ in the lemma as a \emph{fingerprint} -- it is a rather small set that is a good representative of the whole $R$. 

\begin{lemma} \label{lemma:compression}
    Let $G$ be a labelled graph with $n$ vertices such that $(1, 2, \ldots, n-1, n, 1)$ forms a Hamilton cycle, denoted by $H$. Suppose a set of chords $R \subseteq E(G) \setminus E(H)$ has the property that for each two distinct $e, e' \in R$, there is at most one chord $e'' \in R \setminus \{e, e'\}$ such that $|C(e, e')| = |C(e, e'')|$. If $|R| \ge C_0$, where $C_0$ is a sufficiently large constant, then there exists a subset of chords $F \subseteq R$ of size $|F| = \sqrt{|R|}$ such that  $|\cS(H + F)| \ge |R|/24$.
\end{lemma}
\begin{proof}
    Let $F_1 \subseteq R$ be an arbitrary subset of size $\sqrt{|R|}/2$. Set $Y = \mathcal{S}(H + F_1)$ and $F_2 = \emptyset$, and repeat the following $\sqrt{|R|}/2$ times: 
    Pick $e \in R \setminus (F_1 \cup F_2)$ which maximises $| \cS(\mathcal{C}(e, F_1)) \setminus Y|$, add it to $F_2$ and update $Y := Y \cup \cS(\mathcal{C}(e, F_1))$. We claim that if $|Y| < |R| / 6$ at the beginning of some iteration, then in that iteration $Y$ increases by at least $|F_1| / 6$. This implies that taking $F := F_1 \cup F_2$ at the end of the process, we have $|\cS(H + F)| \ge |\cS(\cC(F))| \ge |R|/24$.

    Consider sets $Y$ and $F_2$ at the beginning of some iteration, and suppose $|Y| < |R|/3$. Let $B$ be an auxiliary bipartite graph with vertex sets $R$ and $\mathbb{N} \setminus Y$, and an edge between $e$ and $x$ if there exists $f \in F_1$ such that $|C(e, f)| = x$. Note that the vertices corresponding to $F_1 \cup F_2$ have no incident edges in $B$. Indeed, if $e \in F_1$ and $f \in F_1$ then $|C(e,f)| \in Y$ by the way we initialised $Y$; if $e \in F_2$ and $f \in F_1$, then $|C(e,f)|$ was added to $Y$ in the round where we added $e$ to $F_2$. For an $f \in F_1$, by the assumption of the lemma each $y \in Y$ can be realised by at most two cycles $C(e,f)$ and $C(e',f)$ with $e,e' \in R$. Therefore, there are at least $|R|- 1 - 2|Y| \ge |R|/3$ chords $e \in R \setminus \{f\}$ such that $|C(e, f)| \not \in Y$. As we have just observed, these chords have to lie outside of $F_1 \cup F_2$. Again using the assumption of the lemma, for each $e \in R$ and $f \in F_1$, there is at most one other $f' \in F_1 \setminus \{f\}$ such that $|C(e, f)| = |C(e, f')|$. Therefore, $B$ contains at least $|F_1| |R| / 6$ edges. The quantity $| \cS(\mathcal{C}(e, F_1)) \setminus Y|$ corresponds to the degree of $e$ in $B$. As we pick a chord which maximises this value, it is at least as large as the average degree of vertices from $R$ in the graph $B$, which is at least $|F_1| / 6$. This verifies the claim.
\end{proof}

\section{Container lemmata} \label{sec:containers}

Lemmata presented in this section are the key new ingredients in the proof of Theorem \ref{thm:main}. Informally, they show there exists a small family of large sets which are unavoidable by cycle sets. That is, every graph with certain properties has to contain one of these sets as a subset of its cycle set. We start with the easier of the two results.

\begin{lemma} \label{lemma:container_maxdeg}
    Given $n$ and $p \ge \log^3 n$, there exists a family $\mathcal{F}'(n,p)$ of subsets of $\{1, \ldots, n\}$ such that
    \begin{equation} \label{eq:bound_F_prime}
        \sum_{S \in \mathcal{F}'(n,p)} 2^{-|S|} = 2^{-\Omega(p)},
    \end{equation}
    with the property that if $G$ is a Hamiltonian graph with $n$ vertices and maximum degree at least $p$, then $S \subseteq \mathcal{S}(G)$ for some $S \in \mathcal{F}'(n,p)$.
\end{lemma}
\begin{proof}
    Without loss of generality, we may assume that $G$ a labelled graph such that $(1,2,\ldots,n-1,n,1)$ forms a Hamiltonian cycle, which we denote with $H$. Let $R \subseteq E(G) \setminus E(H)$ be the set of chords incident to a vertex $v$ with the largest degree. Then $|R| \ge p-2$. We may assume $|R| = p-2$ (otherwise remove arbitrary edges from $R$) and $v = 1$. One easily checks that $R$ satisfies the assumption of Lemma \ref{lemma:compression}, thus we can apply it and obtain a subset $F \subseteq R$ of size $|F| = \sqrt{|R|}$ such that $|\mathcal{S}(H + F)| \ge |R|/24 = (p-2)/24$.

    With the previous discussion in mind, let $\mathcal{G}(n,p)$ consist of all graphs on the vertex set $\{1, \ldots, n\}$ of the form $H + F$, where $H$ is the Hamilton cycle $(1, 2, \ldots, n-1, n, 1)$, $|F| = \sqrt{p-2}$ and $|\mathcal{S}(H + F)| \ge (p-2)/24$. As just shown, for every Hamiltonian graph $G$ with a vertex of degree at least $p$ there exists $G' \in \mathcal{G}(n,p)$ such that $G' \subseteq G$, and so $\mathcal{S}(G') \subseteq \mathcal{S}(G)$. Therefore, we can set $\mathcal{F}'(n,p) = \mathcal{F}(\mathcal{G}(n,p))$. The bound \eqref{eq:bound_F_prime} follows from $|\mathcal{F}'(n,p)| \le |\mathcal{G}(n,p)| \le \binom{n}{2}^{\sqrt{p}}$, $|S| \ge (p-2)/24$ for every $S \in \mathcal{F}'(n,p)$, and $p \ge \log^3 n$ (with room to spare).
\end{proof}

The next result applies to a much broader class of Hamiltonian graphs. The proof is also significantly more involved.

\begin{lemma} \label{lemma:count_ham}
    Given a sufficiently large $n$ and $p \ge \log^9 n$, there exists a family $\mathcal{F}(n,p)$ of subsets of $\{1, \ldots, n\}$ such that
    \begin{equation} \label{eq:container}
        \sum_{S \in \mathcal{F}(n,p)} 2^{-|S|} = 2^{-\Omega(\sqrt{p} / \log n)},
    \end{equation}
    with the property that if $G$ is a Hamiltonian graph with $n$ vertices and at least $n + p$ edges, then $S \subseteq \cS(G)$ for some $S \in \mathcal{F}(n,p)$.
\end{lemma}

Note that there exists a Hamiltonian graph $G$ with $n$ vertices and $n+p$ edges such that $|\cS(G)| = \Theta(\sqrt{p})$. For example, start with a Hamilton cycle, choose $\sqrt{p}$ consecutive vertices and add all possible edges between them. Therefore, the property \eqref{eq:container} is the best possible up to the $\log n$ factor in the exponent.

The proof of Lemma \ref{lemma:count_ham} combines ideas of Milans, Pfender, Rautenbach, Regen, and West \cite{milans12spectra}, as well as some ideas of Verstra\"{e}te \cite{verstraete04cycles}, with Lemma \ref{lemma:compression}. The overall theme in the proof is showing that given a graph $G$, one can produce a succinct encoding $\Phi(G)$ and, from it, a large set $\phi(\Phi(G)) \subseteq \mathbb{N}$ with the property $\phi(\Phi(G)) \subseteq \mathcal{S}(G)$. 

\begin{proof}[Proof of Lemma \ref{lemma:count_ham}]
    Let us denote with $\mathcal{H}(n,p)$ the family of all Hamiltonian graphs with $n$ vertices and at least $n + p$ edges. We can suppose that all the graphs in $\mathcal{H}(n,p)$ are labelled and that $(1,2,\ldots,n-1,n,1)$ forms a Hamilton cycle, denoted by $H$. We partition $\cH(n,p)$ into subfamilies $\mathcal{H}_1$ and $\mathcal{H}_2$, and construct a separate family $\mathcal{F}_i$ for each. Throughout the proof, we fix $K$ to be a sufficiently large constant.

    \paragraph{Graphs with many independent chords.} Given a graph $G \in \mathcal{H}(n,p)$, we say that a subset $R \subseteq E(G) \setminus E(H)$ of chords is \emph{independent} if there exists an ordering $e_1, \ldots, e_{|R|}$ of the edges in $R$ such that $b_{e_i} < a_{e_{i+1}}$ for every $i \in \{1, \ldots, |R| - 1\}$. We think of the chords in $R$ as being independent as one can decide for each chord, independently of all other chords, whether it is used to shortcut $H$. We now define $\cH_1$ as follows:
    $$
        \cH_1 = \biggl\{ G \in \mathcal{H}(n,p) \colon \text{there exists an independent } R \subseteq E(G) \setminus E(H) \text{ of size } |R| \ge \sqrt{p} / (K \log n)
        \biggr\}.
    $$
    We show that there exists a mapping
    $$
        \Psi \colon \cH_1 \to (\{0,\ldots,n-1\} \times \{0, \ldots, n-1\})^{t},
    $$
    where $t = p^{1/4}$, with the following properties:
    \begin{enumerate}
        \item Consider some $\mathbf{a} = ((n_1, d_1), \ldots, (n_t, d_t)) \in \Psi(\cH_1)$, and let $D(\mathbf{a})$ be the multiset consisting of $n_i$ many elements $d_i$, for each $i \in [t]$. Then the set
        $$
            \cS(\mathbf{a}) = \left\{ n - \sum_{d \in D'} d \colon D' \subseteq D(\mathbf{a}) \right\} 
        $$
        is of size at least $\sqrt{p} / (K \log n)$.
        \item $\cS(\Psi(G)) \subseteq \cS(G)$ for every $G \in \cH_1$.
    \end{enumerate}
    Having such a mapping, we simply take $\mathcal{F}_1 = \{ \mathcal{S}(\mathbf{a}) \colon \mathbf{a} \in \Psi(\cH_1)\}$ which clearly satisfies desired properties.

    Let us now show how to construct $\Psi$. Consider some $G \in \cH_1$ and let $R \subseteq E(G) \setminus E(H)$ be a set of independent chords of size $|R| = \sqrt{p} / (K \log n)$. If we shortcut $H$ using a chord $e \in R$, we leave out exactly $b_e - a_e - 1$ vertices. More generally, if we start with $H$ and then successively shortcut it using chords from some $R' \subseteq R$, in an arbitrary order, we end up with a cycle of size
    $$
        n - \sum_{e \in R'} (b_e - a_e - 1).
    $$    
    We stress that the order in which we shortcut, as well as the actual values $a_e, b_e$ for chords in $R'$, is irrelevant for the final outcome. Let $D = \{b_e - a_e - 1 \colon e \in R\}$. 
    
    \begin{itemize}
    \item \textbf{Case 1: $|D| \le p^{1/4}$.} Let $d_1, \ldots, d_{|D|}$ be an arbitrary ordering of the elements of the set $D$, and for each $d_i$ let $n_i$ denote the number of chords $e \in R$ such that $b_e - a_e  -  1= d_i$. Set
    $$
        \Psi(G) := ((n_1, d_1), \ldots, (n_{|D|}, d_{|D|}), (0,0), \ldots, (0,0)),
    $$
    where $(0,0)$ is padded $t - |D|$ times (this is just a technicality to make $\Psi$ well defined). Using the observations about short-cutting $H$ with a subset of chords $R' \subseteq R$, we conclude
    $$
        \mathcal{S}(\Psi(G)) = \left\{ n - \sum_{e \in R'} (b_e - a_e - 1) \colon R' \subseteq R \right\} \subseteq \cS(G).
    $$
    From the second set and $b_e \ge a_e + 2$, as otherwise $e$ is not a chord, one easily observes $|\mathcal{S}(\Psi(G))| \ge |R|$.
    
    \item \textbf{Case 2: $|D| > p^{1/4}$.} Choose a subset $F \subseteq R$ of size $|F| = p^{1/4}$ such that $b_e - a_e \neq b_{e'} - a_{e'}$ for distinct $e, e' \in F$. Let $d_1 > \ldots > d_{|F|}$ be the descending ordering of the values $b_e - a_e - 1$ for $e \in F$, and set
    $$
        \Psi(G) := ((1, d_1), \ldots, (1, d_{|F|})).
    $$
    The same reasoning as in the previous case shows $\cS(\Psi(G)) \subseteq \cS(G)$. To see that $\cS(\Psi(G))$ is sufficiently large, for each integer $1 \le q \le |F|$ consider all the sums of the form
    $$
        d_1 + \ldots + d_q + d_i
    $$
    for $i > q$. Then all these sums are different, and there are $\binom{|F|)}{2} = \Omega(\sqrt{p})$ many of them. Note that this is in fact a stronger lower bound than required.
    \end{itemize}

    \paragraph{Graphs with rich chords interaction.} Let $\cH_2 = \cH(n,p) \setminus \cH_1$. We show that there exists a mapping
    $$
        \Phi \colon \cH_2 \to \cG(n),
    $$
    where $\cG(n)$ denotes the family of graphs on the vertex set $\{1, \ldots, n\}$, with the following properties for every $F \in \Phi(\cH_2)$:
    \begin{enumerate}
        \item $|\cS(H + F)| \ge 2 |F| \log n + \Omega( \sqrt{p} / \log n)$, and
        \item If $F = \Phi(G)$, for some $G \in \cH_2$, then $F \subseteq G$.
    \end{enumerate}
    Set $\mathcal{F}_2 = \{\mathcal{S}(H + F) \colon F \in \Phi(\cH_2)\}$. As $H + \Phi(G) \subseteq G$, and consequently $\cS(H + F) \subseteq \cS(G)$, we just need to verify \eqref{eq:container}:
    $$
        \sum_{S \in \mathcal{F}_2} 2^{-|S|} \le \sum_{f = 0}^{\binom{n}{2}} \binom{n}{2}^f 2^{-2f \log n - \Omega(\sqrt{p}/ \log n)} = 2^{-\Omega(\sqrt{p} / \log n)}.
    $$
    
    Consider some $G \in \cH_2$. By the pigeon-hole principle, there exist a set of chords $R \subseteq E(G) \setminus E(H)$ of size $|R| \ge p / \log n$ such that 
    \begin{equation} \label{eq:chord_length}
        L \le b_e - a_e < 2L
    \end{equation}
    for every $e \in R$, where $L$ is some power of two. Let $I \subseteq R$ denote a largest subset of chords which is independent, that is, there is an ordering $e_1, \ldots, e_{|I|}$ of the chords in $I$ such that $b_{e_i} < a_{e_{i+1}}$ for $1 \le i < |I|$. In case there are multiple such largest sets, choose one which minimises
    \begin{equation} \label{eq:minimise}
        \sum_{e \in I} b_e - a_e.
    \end{equation}
    For each $i \in \{1, \ldots, |I|\}$, let $X_i \subseteq R$ be the set consisting of all $e \in R$ such that $a_{e} \le a_{e_i} \le b_e$ or $a_e \le b_{e_i} \le b_e$. As $I$ minimises \eqref{eq:minimise}, each chord belongs to at least one set $X_i$. Let $\ell \in \{1, \ldots, |I|\}$ denote the index of a largest set $X_\ell$, and note that the previous observation implies $|X_\ell| \ge |R| / |I|$. Choose $x \in \{a_{e_\ell}, b_{e_\ell}\}$ such that at least $|X_\ell|/2$ chords $e \in F_\ell$ satisfy $a_e \le x \le b_e$. Let us denote such chords with $X$. Define a relation $\preccurlyeq$ on $X$ as $e \preccurlyeq e'$ if $a_e \le a_{e'}$ and $b_{e'} \le b_{e}$, and note that $(X, \preccurlyeq)$ is a partially ordered set. If $e \preccurlyeq e'$, then we say $e$ and $e'$ are ``parallel'', and otherwise they are ``intersecting''.

    \begin{itemize}
    \item \textbf{Case 1: Many parallel or intersecting chords.} Suppose $X$ contains a chain or an anti-chain $Y$ of size $|Y| = \sqrt{p} / (K \log n)$. In both cases, the set of chords $Y$ satisfies the assumption of Lemma \ref{lemma:compression}. This is easily seen from the definition of $C(e, e')$, and perhaps the only thing worth pointing out is that if $Y$ is an anti-chain, then it is crucial here that $a_e \le x \le b_e$ for every $e \in X$, which forces $a_e < a_{e'} < b_e < b_{e'}$ for $e, e' \in Y$ such that $a_e < a_{e'}$.  Let $F \subseteq Y$ be a subset given by Lemma \ref{lemma:compression}: $|F| = \sqrt{|Y|}$ and $|\cS(H + F)| \ge |Y|/24$. Set $\Psi(G) = F$ (the first property of $\Psi$ follows from the lower bound on $p$).

    \item \textbf{Case 2: Rich combination of intersecting chords with $I$.} Let $Z \subseteq X$ be a largest chain in $X$. By the Erd\H{o}s-Szekeres theorem, there exists an anti-chain $A \subseteq X$ of size at least $|X|/|Z|$. That is, no two chords in $A$ are $\preccurlyeq$-comparable. Again, we stress that in this case we have $a_e < a_{e'} < b_e < b_{e'}$ for every $e, e' \in A$ such that $a_e < a_{e'}$. As the assumption of the previous case is not satisfied, we have $|Z|,|A| < \sqrt{p} / (K \log n)$. From $G \not \in \cH_1$ we have $|I| < \sqrt{p} / (K \log n)$. We show that $|\cS(H + A + I)| \ge \sqrt{p}$, which allows us to set $\Phi(G) = A + I$. In a sense, this is technically the easier of the two cases as it does not use Lemma \ref{lemma:compression}.
    
    From $|A||Z| \ge |X|$, $|X| \ge |R|/(2|I|)$, and the upper bound on $|Z|$, we conclude
    \begin{equation} \label{eq:AI_product}
        |A||I| \ge \frac{K}{2} \sqrt{p}.
    \end{equation}
    As $K$ is sufficiently large, it suffices to show, say, $|\cS(H + A + I)| \ge |A||I|/12$.
    
    Consider $f \in A$ with the smallest $a_f$. Then all the cycles in $\cC(f, A)$ have different size, and so $| \cS(\cC(f, A))| = |A|-1$. Moreover, by the assumption \eqref{eq:chord_length} and the definition \eqref{eq:cycle_def}, we have
    $$
        \cS(\cC(f, A)) \subseteq \{n - 4L, \ldots, n\}. 
    $$
    The main idea from Milans et al.~\cite{milans12spectra} is that one can use chords from $I$ to  ``shift'' sizes of cycles in $\cC(f, A)$ and construct new cycles of different sizes. Owing to \eqref{eq:chord_length}, for every $e \in A$ we have $b_e < a_{e_i}$ for $i \ge \ell + 3$, and $b_{e_j} < a_f$ for $j \le \ell - 3$. That means we can use any subset $S \subseteq \{e_1, \ldots, e_{\ell-3}, e_{\ell+3}, \ldots, e_{|I|} \} := I'$ to shortcut any cycle from $\cC(f, A)$. Let $S_1 \subset S_2 \subset \ldots \subset S_k \subseteq I'$, $k = \lfloor |I'| / 5 \rfloor > |I|/6$ (using that $I$ is sufficiently large, which follows from \eqref{eq:AI_product}), be arbitrary sets such that $|S_1| = 5$ and $|S_i \setminus S_{i-1}| = 5$ for each $i \in \{2, \ldots, k\}$. By short-cutting every $\cC(f, A) =: \cC_0$ using $S_1$, we obtain a set of cycles $\cC_1$ such that
    $$
        \cS(\cC_1) \subseteq \{n - 4L - d_1, n - d_1\},
    $$
    where $d_1 = \sum_{s \in S_1} b_s - a_s - 1 > 4L$. This implies
    $$
        \cS(\cC_1) \cap \cS(\cC_0) = \emptyset.
    $$
    Moreover, short-cutting two cycles of different size using $S_1$ produces two new cycles, again of different sizes. Therefore, $|\cS(\cC_1)| = |A| - 1$.

    Repeating the previous idea, short-cutting cycles in $\cC_0$ using $S_i$ we obtain a set of cycles $\cC_i$ such that
    $$
        \cS(\cC_i) \subseteq \{n - 4L - d_i, n - d_i\},
    $$
    where $d_i = \sum_{s \in S_i} b_s - a_s - 1$. For $0 \le i < j \le k$ we have $d_i < d_j - 4L$ (where $d_0 := 0$), thus
    $$
        \cS(\cC_i) \cap \cS(\cC_j) = \emptyset.
    $$
    Since $|\cS(\cC_i)| = |A|-1$, we conclude 
    $$
        |\cS(H + A + I)| \ge (k+1)(|A| - 1) > |I||A|/12,
    $$
    with some room to spare.
    \end{itemize}
\end{proof}

\section{Proof of Theorem \ref{thm:main}}

The proof of Theorem \ref{thm:main} follows the approach of Verstra\"{e}te \cite{verstraete04cycles} of reducing the problem to counting cycle sets of graphs containing a large induced Hamiltonian subgraph with many chords or a large maximum degree. The latter is taken care of by lemmata from Section \ref{sec:containers}. This is the part that is done more efficiently than in \cite{verstraete04cycles}, and is the sole source of the improvement.

\begin{proof}[Proof of Theorem \ref{thm:main}]
The quantity we are interested in is the size of the family $\cS(\mathcal{G}(n))$, where $\mathcal{G}(n)$ denotes the family of all non-isomorphic $n$ vertex labelled graphs (that is, graphs on the vertex set $\{1, \ldots, n\}$). Following Verstra\"{e}te \cite{verstraete04cycles}, consider the following families of graphs:
\begin{align*}
    \mathcal{G}_1 &= \{G \in \mathcal{G}(n) \colon \text{ a largest cycle in $G$ is of size at most } n - \sqrt{n}/(4 \log n)\}, \\[1em]
    \mathcal{G}_2 &= \{G \in \mathcal{G}(n) \setminus \mathcal{G}_1 \colon G \text{ has at most } n + n / (4 \log n) \text{ edges}\}, \\[1em]
    \mathcal{G}_3 &= \biggl\{G \in \mathcal{G}(n) \colon \parbox{7.2cm}{$G$ contains an induced Hamiltonian subgraph \\ 
    with maximum degree at least $\sqrt{n} / 4 $} \biggr\},  \\[1em]
    \mathcal{G}_4 &= \biggl\{G \in \mathcal{G}(n) \colon \parbox{7.7cm}{$G$ contains an induced Hamiltonian subgraph $G'$ \\
    with at least $v(G') + n / (8 \log n)$ edges} \biggr\}.
\end{align*}
We first show $\cG(n) = \cG_1 \cup \cG_2 \cup \cG_3 \cup \cG_4$. Consider some $G \in \cG(n) \setminus (\cG_1 \cup \cG_2)$. Then $G$ contains a cycle $C = (c_1, c_2, \ldots, c_\ell)$ of size $\ell \ge n - \sqrt{n} / (4 \log n)$, and has at least $n + n / (4 \log n)$ edges. Suppose there exists a vertex $v \in V(G) \setminus V(C)$ with at least $\sqrt{n} / 4$ neighbours in $C$. Let $i$ denote the smallest index of a vertex $c_i$ which is neighbour of $v$, and $j$ the largest index. Then $C' = (v, c_i, \ldots, c_j)$ forms a cycle, and $v$ has degree at least $\sqrt{n}/4$ in the induced subgraph $G[C']$. Therefore, $G \in \cG_3$. Otherwise, if no vertex $v \in V(G) \setminus V(C)$ has this property, then the induced subgraph $G[C]$ has at least
$$
    n + \frac{n}{4 \log n} - \frac{\sqrt{n}}{4 \log n} \left(\frac{\sqrt{n}}{4 \log n} + \frac{\sqrt{n}}{4} \right) > n + \frac{n}{8 \log n}
$$
edges, with room to spare. As $|C| \le n$, we conclude $G \in \cG_4$.

We estimate the size of each $\cS(\cG_i)$ separately, and then use an upper bound
 $$
    |\cS(\cG(n))| \le |\cS(\cG_1)| + |\cS(\cG_2)| + |\cS(\cG_3)| + |\cS(\cG_4)|.
 $$

\begin{itemize}
\item From $\cS(G) \subseteq \{3, \ldots, n - \sqrt{n} / (4 \log n)\}$ for $G \in \cG_1$, we trivially have 
$$
    |\cS(\cG_1)| \le 2^{n - \sqrt{n} / (4 \log n)}.
$$

\item We use $|\cS(\cG_2)| \le |\cG_2|$. Every graph in $\cG_2$ can be obtained by starting with a cycle of size $n \ge \ell \ge n - x$, for some $x < \sqrt{n} / (4 \log n)$, and then adding remaining $r \le n / (4 \log n) + x < n / (3 \log n)$ edges. This can be done in at most $\binom{n}{2}^r < n^{2r} < 2^{2n/3}$ ways,  which gives us
$$
    |\cS(\cG_2)| < n^2 2^{2n/3}. 
$$

\item Set $p = \sqrt{n}/4$, and let $\cF' = \bigcup_{m = 1}^n \cF'(m, p)$, where $\cF'$'s are given by Lemma \ref{lemma:container_maxdeg}. Then for every $G \in \cG_3$ there exists $S \in \cF'$ such that $S \subseteq \cS(G') \subseteq \cS(G)$, where $G'$ is a Hamiltonian subgraph of $G$ with maximum degree at least $p$. Therefore, the family of up-sets of sets in $\cF'$ contains all possible cycle sets of graphs in $\cG_3$. This gives us the following:
$$
    |\cS(\cG_3)| \le \sum_{m = 1}^n \sum_{S \in \mathcal{F}'(m, p)} 2^{n - |S|} \stackrel{\eqref{eq:bound_F_prime}}{\le} n 2^{n - \Omega(\sqrt{n})}.
$$
Note that this is actually a stronger bound than needed.

\item The bound on $\cS(\cG_4)$ is obtained analogously, with Lemma \ref{lemma:container_maxdeg} replaced by Lemma \ref{lemma:count_ham}:
$$
    |\cS(\cG_4)| \le \sum_{m = 1}^n \sum_{S \in \cF(m, \frac{n}{8 \log n})} 2^{n - |S|} \stackrel{\eqref{eq:container}}{\le} 2^{n - \Omega(\sqrt{n} / \log^{3/2}(n))}.
$$
\end{itemize}
\end{proof}

\section{Concluding remarks} \label{sec:concluding}

We believe that the logarithmic factor in Theorem \ref{thm:main} can be somewhat optimised. It would be interesting to fully remove it, which at present seems difficult. Even if one could get an optimal bound $2^{-\Omega(\sqrt{p})}$ in Lemma \ref{lemma:count_ham}, which is a challenge on its own, it is not clear how to make use of it to reach $2^{n - \Omega(\sqrt{n})}$. In particular, the main bottleneck to further improvements seems to be the family $\cG_2$.  Let us briefly describe the issue.

Like Verstra\"{e}te \cite{verstraete04cycles}, we take care of $\cG_2$ by using a simple bound $|\cS(\cG_2)| \le |\cG_2|$. To make this upper bound non-trivial, the bound on the number of edges of the form $n + O(n / \log n)$ is crucial in the definition of $\cG_2$. Raising the threshold on the number of edges in $\cG_4$  to $n + \eps n$ would, with an optimal version of Lemma \ref{lemma:count_ham}, indeed give us a desired bound on the size of $\cS(\cG_4)$. However, this requires raising the threshold in the definition of $\cG_2$ as well, at which point handling $\cG_2$ becomes difficult -- too few edges to effectively use our container lemma, but enough to make the size of $\cG_2$ too large. 

As we have seen in Lemma \ref{lemma:container_maxdeg}, having a few edges is not per se a bottleneck if we know something about the structure of a graph. Unfortunately, the assumption of Lemma \ref{lemma:container_maxdeg} is rather specific and does not provide any ideas how to proceed further. To summarise, we see a better handling of the family consisting of graphs with a long cycle and (only) linear number of edges as a likely necessary step for any further advancement, at least for approaches  along the lines of the presented proof.

\bibliographystyle{abbrv}
\bibliography{cycle_sets}

\end{document}